\documentclass{amsart}
 
\usepackage{amsthm,amsmath,amssymb,graphicx}
\RequirePackage[dvips]{hyperref}
\RequirePackage{hypernat}

\newtheorem{theorem}{Theorem}
\newtheorem{lemma}{Lemma}
\newtheorem{proposition}{Proposition}

\newtheorem{conj}{Conjecture}

\theoremstyle{definition}

\newtheorem{example}{Example}

\theoremstyle{remark}

\allowdisplaybreaks[1]

\title[Variance component models]{Maximum likelihood
  degree of variance component models}

\author[Gross, Drton, Petrovi\'c]{Elizabeth Gross, Mathias Drton \and Sonja
  Petrovi\'c} 

\begin{document}

\begin{abstract}
  Most statistical software packages implement numerical strategies
  for computation of maximum likelihood estimates in random effects
  models.  Little is known, however, about the algebraic complexity of
  this problem.  For the one-way layout with random effects and
  unbalanced group sizes, we give formulas for the algebraic degree of
  the likelihood equations as well as the equations for restricted
  maximum likelihood estimation.  In particular, the latter approach
  is shown to be algebraically less complex.  The formulas are
  obtained by studying a univariate rational equation whose solutions
  correspond to the solutions of the likelihood equations.  Applying
  techniques from computational algebra, we also show that balanced
  two-way layouts with or without interaction have likelihood
  equations of degree four.  Our work suggests that algebraic methods
  allow one to reliably find global optima of likelihood functions of
  linear mixed models with a small number of variance components.
\end{abstract}

\maketitle

\section{Introduction}
\label{sec:intro}

Linear models with fixed and random effects are widely used for
dependent observations.  Such mixed models are typically fit using
likelihood-based techniques, and the necessary optimization problems
can be solved using the numerical methods implemented in various
statistical software packages, as discussed, for instance, in
\cite{Faraway:2006}.  Such software typically takes into account that
the variance parameters are nonnegative.  However,
general-purpose optimization procedures do not give any guarantees
that a global optimum is found; compare Section 1.8 in
\cite{jiang:2007}.  It can thus be appealing to compute maximum
likelihood (ML) estimates algebraically.  Since linear mixed models
have rational likelihood equations, this involves careful clearing of
denominators and applying symbolic and specialized numerical
techniques to determine all solutions of the resulting polynomial
system.  An explanation of what we mean by careful clearing of
denominators is given in \cite[Chap.~2]{Drton:book}.  While solving
likelihood equations algebraically may not be feasible in large models
with several random factors, modern computational algebra does allow
one to fully understand the likelihood surface in practically relevant
settings.

The main contribution of this paper is a study of the algebraic
complexity of ML estimation in the unbalanced one-way layout with
random effects.  This model concerns a collection of grouped
observations
\begin{equation}
  \label{eq:model-oneway}
Y_{ij} = \mu+\alpha_i + \varepsilon_{ij}, \qquad i=1,\dots,q, \quad
j=1,\dots,n_i.
\end{equation}
The overall mean $\mu\in\mathbb{R}$ is a fixed (`non-random') but
unknown parameter.  The random effects $\alpha_i$ and the error terms
$\varepsilon_{ij}$ are mutually independent normal random variables.
More precisely, $\alpha_i\sim\mathcal{N}(0,\tau)$ and
$\varepsilon_{ij}\sim\mathcal{N}(0,\omega)$, where $\tau$ and $\omega$
denote the common variances of the random effects and the error terms,
respectively.  Clearly, the distribution of observation $Y_{ij}$ is
$\mathcal{N}(\mu,\tau+\omega)$, and two observations $Y_{ij}$ and
$Y_{ik}$ from the $i$th group are dependent with covariance $\tau$.  A
detailed discussion and examples of applications of this specific
model can be found, for instance, in Chapter 3 of \cite{Searle:1992}
and in Chapter 11 of \cite{Sahai:II}.

The covariance matrix of the joint multivariate normal distribution
for all $Y_{ij}$ defined by (\ref{eq:model-oneway}) is the product of
the scalar $\omega$ and a matrix that is a function of the variance
ratio $\theta=\tau/\omega$.  Therefore, when $\theta$ is known, the
likelihood equations for $\mu$ and $\omega$ are of the type
encountered in generalized least squares calculations, with a unique
solution that is a rational function of the data and the known value
of $\theta$.  We may thus eliminate $\mu$ and $\omega$ from the
likelihood equations, which then reduce to a single univariate
equation.  Before turning to a first example, we remark that we always
tacitly assume suitable sample size conditions to be satisfied such
that ML estimates exist.  In particular, we assume there to be $q\ge
2$ groups with at least one group of size $n_i\ge 2$.  A definitive
answer to the existence problem in linear mixed models is given in
\cite{Demidenko:1999} who also treat restricted maximum likelihood
(REML) estimation; see \cite{McCullagh:1989} for an introduction to
this technique.

\begin{example}
  \label{ex:dyestuff}
  Textbook data from \cite[\S6.4]{Davies:1972} give the yield of
  dyestuff from $5$ different preparations from each of $q=6$
  different batches of an intermediate product; the data are also
  available in the R package {\tt lme4}.  The layout is balanced, that
  is, all batch sizes are equal, here $n_i=5$.  In this case, the
  likelihood equations are well-known to be equivalent to a linear
  equation system, and the ML estimators are rational functions of the
  observations $Y_{ij}$.  In terminology we will use later on, balanced
  one-way layouts have ML degree one.  Exactly the same is true for
  REML.
  
  A different picture emerges in the unbalanced case, when the batch
  sizes are not all equal.  For illustration, we remove the first,
  second and sixth observation from the data.  The first batch then
  only comprises $n_1=3$ preparations, and the second batch only
  $n_2=4$.  The remaining batches are unchanged with $n_i=5$ for $i\ge
  3$.  In this unbalanced case, the solutions of the likelihood
  equations correspond to the solutions of the polynomial equation
  \begin{multline}
    \label{eq:ml-dyestuff}
    -245488320000\,\theta^7 -277109078400\,\theta^6 -58814614680\,\theta^5
    +54052612853\,\theta^4 \\
    +37792395524\,\theta^3 +10086075110\,\theta^2
    +1279832076\,\theta+64175517 = 0.
  \end{multline}
  As the large integers suggest, this equation is exact given the
  input.  However, the measurements that enter the computation are, of
  course, rounded.
  
  Numerical optimization using the R package {\tt lme4} yields a local
  maximum of the likelihood function that corresponds to
  $\theta\approx 0.5585$.  We may check whether this local maximum is
  unique, or at least a global optimum, by finding all roots of the
  above univariate polynomial.  This is a task that can be done
  reliably in computer algebra systems.  However, such a computation
  is not needed here.
  The polynomial in (\ref{eq:ml-dyestuff}) has exactly one sign change
  in its coefficient sequence.  Hence, by Descartes' rule of signs, it
  has precisely one positive real root.  The mere construction of the
  polynomial thus reveals that the local maximum we computed is the
  unique local (and global) maximum of the likelihood function; recall
  that the parameter $\theta$ is restricted to be nonnegative.
  
  A similar story unfolds for REML estimation.  The only difference is
  that the degree of the relevant polynomial drops to five:
  \begin{multline}
    \label{eq:reml-dyestuff}
    -17047800000\,\theta^5-6811774200\,\theta^4
    +5505084700\,\theta^3\\ 
    +4048254212\,\theta^2
    +897954164\,\theta+67458244 = 0.
  \end{multline}
  We note that equations (\ref{eq:ml-dyestuff}) and
  (\ref{eq:reml-dyestuff}) cannot be solved by radicals.  The Galois
  groups are the symmetric groups $S_7$ and $S_5$, respectively.  
\end{example}

It is natural at this point to ask for the {\em maximum likelihood
  degree} of the one-way layout as a function of the number of groups
$q$ and the group sizes $n_1,\dots,n_q$.  The ML degree is the number
of complex solutions to the likelihood equations when the data are
{\em generic}.  Indeed, the number of complex solutions is constant
with probability one, and a data set is generic if it is not part of
the null set for which the number of complex solutions is different.
The REML degree is defined in just the same way, but starting
from different equations.
Either degree measures the algebraic complexity of the computation of
the estimates.  In Example~\ref{ex:dyestuff}, it is simply the degree
of a univariate polynomial in $\theta=\tau/\omega$ whose roots yield
the possibly complex vectors $(\mu,\omega,\tau)$ that
solve the likelihood equations.  For more background on ML degrees,
see
\cite{Hosten:2005,Catanese:2006,Buot:2007,Drton:book,Sturmfels:2009,Hosten:2010}. 

Our main result answers the above question.
Theorem~\ref{thm:MAIN}, which we prove in later sections, gives
formulas for both the ML and the REML degree of possibly unbalanced
one-way layouts and offers a direct comparison of the algebraic
complexity of the two approaches.  The theorem is conveniently stated
using a notion of multiplicities.  Suppose
$v=(v_1,\dots,v_q)\in\mathbb{Z}^q$ is a tuple of integers.  If $v$ has
$M$ distinct entries, then the multiplicities of $v$ form the integer
multiset $\{m_1,\dots,m_M\}$, where $m_j$ counts how often the $j$th
distinct entry of $v$ appears among all entries of $v$.

\begin{theorem}
  \label{thm:MAIN}
  Consider a one-way layout with random effects for $q$ groups that
  are of sizes $n_1,\dots,n_q$.  Suppose $M$ of the group sizes are
  distinct, with associated multiplicities $m_1,\dots,m_M$.  Let
  $M_2=\#\{j: m_j\ge 2\}$.  Then the ML degree is $3M+ M_2-3$, and the
  REML degree is $2M+2M_2-3$.  The ML degree exceeds the REML degree
  unless $M_2=M$, in which case equality holds.
\end{theorem}

The condition $M_2=M$ holds if each group size appears at least twice.
In the balanced case, we have $M=M_2=1$ and the theorem recovers the
well-known fact that both degrees are one; compare
\cite{Hocking:1985,Searle:1992,Sahai:I}.  Each degree is maximal when
the group sizes $n_1,\dots,n_q$ are pairwise distinct.  The degrees
are then $3q-3$ for ML and $2q-3$ for REML.

\begin{example}
  The model for the dyestuff data from Example~\ref{ex:dyestuff} has
  $q=6$ groups.  The unbalanced case we considered had group sizes
  $(n_1,\dots,n_6)=(3,4,5,5,5,5)$.  The multiplicities are
  $\{1,1,4\}$.  Our formulas confirm the ML and REML degree to be
  $3\cdot 3 + 1 - 3 = 7$ and $2\cdot 3 +2\cdot 1 - 3 = 5$,
  respectively.  As another example, if $(n_1,\dots,n_6)
  =(4,4,3,2,2,2)$, then the ML degree is $8$ and the REML degree is
  $7$.
\end{example}

The remainder of the paper is structured as follows.  In
Section~\ref{sec:deriv-likeqn}, we review the derivation of the
likelihood equations for ML and REML estimation.
Section~\ref{sec:MLproof} contains the proof of the ML degree formula
from Theorem~\ref{thm:MAIN}, and Section~\ref{sec:REMLproof} treats
the REML degree.  Each proof consists of a detailed study of a
univariate rational equation in the variance ratio $\theta$.  In
Section~\ref{sec:general-mean}, we demonstrate that algebraic
computations are feasible for more general linear mixed models.  More
precisely, we treat a one-way layout with $q=109$ unbalanced groups
and a mean structure given by two covariates that is relevant in a
recent application.  In Section~\ref{sec:two-way}, we consider
balanced two-way layouts.  These are known to have REML degree equal
to one, and we show that the ML degree is four, which means that ML
estimates are available in closed form in the sense of Cardano's
formula.  Our conclusions are summarized in
Section~\ref{sec:conclusion}, where we also give two examples of
unbalanced one-way random effects models with bimodal likelihood
functions.

%%%%%%%%%%%%%%%%%%%%%% 
\section{The likelihood equations}
\label{sec:deriv-likeqn}

Let $n_1, \ldots, n_M$ be unique group sizes with associated
multiplicities $m_1, \ldots, m_M$. Let
$Y_{ij}=(Y_{ij1},\dots,Y_{ijn_i})$ be the vector comprising the
observations in the $j$th group of size $n_i$.  Then the model for the
one-way layout given by (\ref{eq:model-oneway}) can equivalently be
described as stating that $Y_{11},\dots, Y_{1m_1}, Y_{21}, \ldots,
Y_{Mm_M}$ are independent multivariate normal random vectors with
\[
Y_{ij} \sim \mathcal{N}\left( \mu\mathbf{1}_{n_i}, \Sigma_{n_i}(\omega,\tau)
\right),
\]
where the covariance matrix is
\[
\Sigma_{n_i}(\omega,\tau) =\omega I_{n_i} +
\tau\mathbf{1}_{n_i}\mathbf{1}_{n_i}^T.
\]
Here, $\mathbf{1}_n=(1,\dots,1)^T\in\mathbb{R}^n$, and $I_n$ is the
$n\times n$ identity matrix.

\subsection{Maximum likelihood}
\label{sec:MLequations}

Ignoring additive constants and multiplying by two, the log-likelihood
function of the one-way model is
\begin{align}
  \label{eq:ell0}
  \ell(\mu,\omega,\tau) 
  & = \sum_{i=1}^M \sum_{j=1}^{m_i} \log\det\left( K_{n_i}(\omega,\tau)\right) -
  (Y_{ij}-\mu\mathbf{1}_{n_i})^T K_{n_i}(\omega,\tau)
  (Y_{ij}-\mu\mathbf{1}_{n_i}),
\end{align}
where 
\begin{equation}
  \label{eq:Sigma-inverse}
  K_{n_i}(\omega,\tau) = 
  \frac{1}{\omega}I_{n_i} - 
  \frac{\tau}{\omega(\omega+n_i\tau)}
  \mathbf{1}_{n_i}\mathbf{1}_{n_i}^T.
\end{equation}
is the inverse of $\Sigma_{n_i}(\omega,\tau)$.  The inverse has
determinant
\begin{equation}
  \label{eq:Sigma-det}
  \det( K_{n_i}(\omega,\tau)) = \frac{1}{\omega^{n_i-1}(\omega+n_i\tau)}.
\end{equation}

Let $N=m_1n_1+\dots+m_Mn_M$ be the total number of
observations.
For each $i=1,\dots,M$, define the {\em group averages\/}
\[
\bar Y_{ij} = \frac{1}{n_i} \sum_{k=1}^{n_i} Y_{ijk}, \quad j=1, \ldots, m_i,
\]
and the average across the groups of equal size
\[
\bar Y_ i =\frac{1}{ m_i} \sum_{j=1}^{m_i} \bar Y_{ij}.
\]
From the averages, compute the {\em between-group sum of squares\/}
\[
B_i = \sum_{j=1}^{m_i} (\bar Y_{ij} - \bar Y_i)^2.
\]
Note that, for generic data, $B_j=0$ if and only if $m_j=1$.
Therefore, it suffices to consider the sums of squares $B_i$ with
$m_i\geq 2$.  Finally, define the {\em within-group sum of squares\/}
\[
W = \sum_{i=1}^M \sum_{j=1}^{m_i} \sum_{k=1}^{n_i} (Y_{ijk}- \bar Y_{ij})^2 ,
\]
which is positive for generic data.

\begin{proposition}
  \label{prop:ell}
  Upon the substitution $\kappa=1/\omega$ and $\theta=\tau/\omega$,
  the log-likelihood function for the one-way layout can be written as
  \begin{multline}
    \label{em:ml-log-lik}
    \ell(\mu, \kappa,\theta) = N\log(\kappa) -\kappa W - \left[
      \sum_{i=1}^M m_i \log(1+n_i\theta) \right]-\kappa
    \left[\sum_{i=1}^M
      \frac{n_i}{1+n_i\theta} B_i \right]\\
     -\kappa
    \left[\sum_{i=1}^M
      \frac{m_in_i}{1+n_i\theta} (\bar Y_i -
      \mu)^2 \right].
  \end{multline}
\end{proposition}
\begin{proof}
  Applying (\ref{eq:Sigma-inverse}), the quadratic form in
  (\ref{eq:ell0}) can be expanded into
  \begin{align*}
     &(Y_{ij}-\mu\mathbf{1}_{n_i})^T K_{n_i}(\omega,\tau)
    (Y_{ij}-\mu\mathbf{1}_{n_i})\\
    &=
    \frac{1}{\omega} \sum_{k=1}^{n_i} ( Y_{ijk}-\mu)^2  -
    \frac{\tau}{\omega(\omega+n_i\tau)}
    \left[(Y_{ij}-\mu\mathbf{1}_{n_i})^T\mathbf{1}_{n_i}\right]^2\\
    &=
    \frac{1}{\omega} \sum_{k=1}^{n_i} ( Y_{ijk}-\bar Y_{ij})^2 +
    \frac{n_i}{\omega}  ( \bar Y_{ij}-\mu)^2 -
    \frac{\tau}{\omega(\omega+n_i\tau)}
    n_i^2( \bar Y_{ij}-\mu)^2\\
    &= \kappa
    \frac{n_i}{1+n_i\theta}
    ( \bar Y_{ij}-\mu)^2
    +\kappa \sum_{k=1}^{n_i} ( Y_{ijk}-\bar Y_{ij})^2 .
  \end{align*}
  Using this expression and (\ref{eq:Sigma-det}), the log-likelihood
  function is seen to be equal to
  \begin{align}
    \label{eq:ell-sub}
    &\ell(\mu,\kappa,\theta)  \\
    \nonumber
    &= N\log(\kappa) -\kappa W-
    \left[\sum_{i=1}^M m_i \log(1+n_i\theta) \right] - 
    \kappa\left[\sum_{i=1}^M \sum_{j=1}^{m_i} \frac{n_i}{1+n_i\theta} (
      \bar Y_{ij}-\mu)^2
    \right] .
  \end{align}
  The claimed form of $\ell(\mu,\kappa,\theta)$ is now obtained by
  expanding the last sum as
  \begin{align}
    & \sum_{j=1}^{m_i} \frac{n_i}{1+n_i\theta}  (
    \bar Y_{ij}-\mu)^2\\
    &= \frac{n_i}{1+n_i \theta} \sum_{j=1}^{m_i}
    \left[(\bar Y_{ij} - \bar Y_i)^2+(\bar Y_i - \mu)^2+2(\bar
    Y_{ij}-\bar Y_i)(\bar Y_i - \mu) \right]\\ 
    &= \frac{m_in_i}{1+n_i \theta} (\bar Y_i -
    \mu)^2  +  \frac{n_i}{1+n_i \theta} B_i.
    \qedhere
  \end{align}
\end{proof}

The partial derivatives of the log-likelihood function from
Proposition~\ref{prop:ell} are
\begin{align}
  \frac{\partial\ell }{\partial\mu} &=
  2\kappa\sum_{i=1}^M\frac{m_i n_i}{1+n_i\theta} (
    \bar Y_i-\mu),\\
  \frac{\partial\ell}{\partial\kappa} &=
  \frac{N}{\kappa} - \left[ W + \sum_{i=1}^M\frac{m_i n_i}{1+n_i\theta} (
    \bar Y_i-\mu)^2  +\sum_{i=1}^M \frac{n_i}{(1+n_i\theta)}B_i\right],\\
  \frac{\partial\ell}{\partial\theta} &=
  - \left[\sum_{i=1}^M \frac{m_i n_i}{1+n_i\theta} \right] + 
  \kappa\left[\sum_{i=1}^M\frac{m_i n_i^2}{(1+n_i\theta)^2} (
    \bar Y_i-\mu)^2+\sum_{i=1}^M\frac{ n_i^2}{(1+n_i \theta)^2}B_i
  \right].
\end{align}
Since $N\not=0$, the equation system obtained by setting the three partials
to zero has the same solution set as the equation system
\begin{align}
  \label{eq:mu}
  \sum_{i=1}^M\frac{m_i n_i}{1+n_i\theta} (
    \bar Y_i-\mu)&=0,\\
  \label{eq:kappa}
  N - \kappa \left[ W + \sum_{i=1}^M\frac{m_i n_i}{1+n_i\theta} (
    \bar Y_i-\mu)^2  + \sum_{i=1}^M\frac{n_i}{1+n_i\theta} B_i \right] &= 0,\\
  \label{eq:theta}
  \kappa\left[\sum_{i=1}^M\frac{m_i n_i^2}{(1+n_i\theta)^2} (
    \bar Y_i-\mu)^2 +\sum_{i=1}^M\frac{n_i^2}{(1+n_i\theta)^2} B_i
  \right] - \left[\sum_{i=1}^M \frac{m_i n_i}{1+n_i\theta} \right] &=0 .
\end{align}
Now we can solve equation~(\ref{eq:mu}) for $\mu$, substitute the
result into equation~(\ref{eq:kappa}) and solve for $\kappa$.  Both
$\mu$ and $\kappa$ are then expressed in terms of $\theta$.
Substituting the expressions into (\ref{eq:theta}), we obtain a
univariate rational equation in $\theta$.  Our proof of the ML degree
formula in Theorem~\ref{thm:MAIN} proceeds by cancelling terms from
the numerator and denominator of this rational expression.  This is
the topic of Section~\ref{sec:MLproof}.

%%%%%%%%%%%%%%%%%%%%%%%%%
\subsection{Restricted maximum likelihood}
\label{sec:REMLequations}

The REML method uses a slightly different likelihood function that is
obtained by considering a projection of the observed random array
$(Y_{ijk})\in\mathbb{R}^N$.  The mean of this array has all entries
equal to $\mu$.  In other words, it is modelled to lie in the space
$\mathcal{L}\subset \mathbb{R}^N$ spanned by the array with all
entries equal to one.  The likelihood function used in REML is
obtained by taking the observation to be the projection of $(Y_{ijk})$
onto the orthogonal complement of $\mathcal{L}$.  The distribution of
the projection no longer depends on $\mu$ and so the REML function
only has $(\tau,\omega)$ or, equivalently, $(\kappa,\theta)$ as
arguments.

Using the formulas given, for instance, in \cite{McCullagh:1989}, and
simplifying the resulting expressions similar to what was done in the
proof of Proposition~\ref{prop:ell}, we obtain the following
expression for the restricted log-likelihood function.

\begin{proposition}
  \label{prop:ell-reml}
  Upon the substitution $\kappa=1/\omega$ and $\theta=\tau/\omega$,
  the restricted log-likelihood function for the one-way layout can be
  written as
  \begin{multline}
    \label{eq:reml-log-lik}
    \bar\ell(\kappa,\theta) = (N-1)\log(\kappa) -\kappa W - \left[
      \sum_{i=1}^M m_i \log(1+n_i\theta) \right]\\
    -\log\left(\sum_{i=1}^M \frac{m_i n_i}{1+n_i\theta} \right) -\kappa
    \left[\sum_{i=1}^M
      \frac{n_i}{1+n_i\theta} B_i \right]-
    \kappa
    \left[\sum_{i=1}^M
      \frac{m_in_i}{1+n_i\theta} (\bar Y_i -
      \hat\mu(\theta))^2 \right],
  \end{multline}
  with
  \begin{equation}
    \label{eq:mu-sol}
    \hat\mu(\theta) = \frac{\sum_{i=1}^M \sum_{j=1}^{m_i}
      \frac{n_i}{1+n_i\theta}  \bar Y_{ij}}{
      \sum_{i=1}^M \sum_{j=1}^{m_i}\frac{n_i}{1+n_i\theta} } = 
    \frac{\sum_{i=1}^M \frac{ m_in_i}{1+n_i\theta} 
      \bar Y_{i}}{
      \sum_{i=1}^M \frac{ m_in_i}{1+n_i\theta} }.
  \end{equation}
\end{proposition}

Note that $\hat\mu(\theta)$ is the solution to the equation in
(\ref{eq:mu}).  Computing $\hat\mu(\theta)$ is the standard way to
obtain an estimate of $\mu$ from a REML estimate of $\theta$.

The partial derivatives of the restricted log-likelihood function from
Proposition~\ref{prop:ell-reml} are 
\begin{align}
  \label{eq:reml-partial-kappa}
  \frac{\partial\bar\ell}{\partial\kappa} &=
  \frac{N-1}{\kappa} - \left[ W + \sum_{i=1}^M\frac{m_in_i}{1+n_i\theta} (
    \bar Y_i-\hat\mu(\theta))^2  +\sum_{1=1}^M \frac{n_i}{1+n_i \theta} B_i\right],\\
  \label{eq:reml-partial-theta}
  \frac{\partial\bar\ell}{\partial\theta} &= 
  - \left[\sum_{i=1}^M \frac{m_i n_i}{1+n_i\theta} \right] 
  + \frac{  \sum_{i=1}^M\frac{m_i n_i^2}{(1+n_i\theta)^2} }{
    \sum_{i=1}^M\frac{m_i n_i}{(1+n_i\theta)} }\\
\notag &+ 
  \kappa\left[\sum_{i=1}^M\frac{m_in_i^2}{(1+n_i\theta)^2} (
    \bar Y_i-\hat\mu(\theta))^2 + \sum_{i=1}^M \frac{n_i^2}{(1+n_i \theta)^2} B_i
  \right].
\end{align}
The equation $\partial\bar\ell/\partial\kappa =0$ is easily solved.
Substituting the unique solution $\hat\kappa(\theta)$ into the
equation $\partial\bar\ell/\partial\theta =0$ yields again a
univariate rational equation in $\theta$.  The proof of the REML
degree formula in Theorem~\ref{thm:MAIN} requires studying
cancellations from the numerator and denominator of this equation,
which is the topic of Section~\ref{sec:REMLproof}.

\bigskip
%%%%%%%%%%%%%%%%%%%%%%%%%%%%%%%
\section{Proof of formula for ML degree}
\label{sec:MLproof}

Our proof of the ML degree formula in Theorem~\ref{thm:MAIN} proceeds
in two steps. First, in Lemma~\ref{lemma:univRatlEqn} we derive a
univariate rational equation whose number of zeros is the ML degree of
the model. Second, we simplify it in Lemmas~\ref{claim1}
and~\ref{claim2} by clearing common factors from the numerator and the
denominator.

Fix the  following notation, used throughout. 
For a vector $a=(a_1,\dots,a_M)\in\mathbb{R}^M$, define the rational
functions
\[
r_a(\theta)= \sum_{i=1}^M \frac{m_i n_i}{1+n_i\theta} a_i \phantom{xx}\mbox{ and }\phantom{xx} s_a(\theta)= \sum_{i=1}^M \frac{m_i n_i^2}{(1+n_i\theta)^2} a_i.
\]
We write $r_1$, $r_{B/m}$, $r_Y$, $r_{Y^2}$ for the functions $r_a$
that have
\[
a=\mathbf{1}_M, \quad a=\left(\frac{B_1}{m_1}, \dots, \frac{B_M}{m_M}\right),
\quad a=(\bar Y_1,\dots,\bar Y_M), \quad a=(\bar
Y_1^2,\dots,\bar Y_M^2),
\]
respectively.  It is clear from Section~\ref{sec:deriv-likeqn} that
forming a common denominator for the rational equations to be studied
involves the product
\[
 d(\theta) = \prod_{i=1}^M (1+n_i\theta)=d_1(\theta)d_2(\theta),
\]
where
\[
d_1(\theta)= \prod_{\{i: m_i=1\}}(1+n_i \theta), \quad
d_2(\theta)=\prod_{\{i: m_i \geq 2\}} (1+n_i \theta).
\]
For a vector $a\in\mathbb{R}^M$, define the degree $M-1$ polynomial
\[
f_a(\theta)= d(\theta)r_a(\theta) = \sum_{i=1}^M m_i n_i a_i \prod_{j\not=i}
(1+n_j\theta) 
\]
and the degree $2(M-1)$
polynomial
\[
g_a(\theta)= d(\theta)^2s_a(\theta) = \sum_{i=1}^M m_i n_i^2 a_i
\prod_{j\not=i} (1+n_j\theta)^2.
\]

\begin{lemma}\label{lemma:univRatlEqn}
  The ML degree of the one-way layout is the degree of the numerator
  created when cancelling all common factors from numerator and
  denominator of the following rational function in $\theta$:
  \begin{multline}
    \label{eq:univar4}
    \frac{1}{Nd(\theta)^2f_1(\theta)^2} \\
    \times\left( N\left[
        f_1(\theta)^2g_{Y^2}(\theta)-2f_Y(\theta)f_1(\theta)g_Y(\theta)+
        f_Y(\theta)^2g_1(\theta)+f_1(\theta)^2g_{B/m}(\theta) \right] \right.\\
    \left.  - f_1(\theta)^2 \left[W f_1(\theta)d(\theta) +
        f_{Y^2}(\theta)f_1(\theta) - f_Y(\theta)^2
        +f_1(\theta)f_{B/m}(\theta) \right] \right).
  \end{multline}
\end{lemma}

\begin{proof}
  Adopting the notation above, the solution of the first of the
  likelihood equations in \eqref{eq:mu} can be written as
  \begin{equation}
    \label{eq:mu-sol2}
    \hat\mu(\theta) = 
    \frac{r_Y(\theta)}{r_1(\theta)}.
  \end{equation}
  Next, rewrite the following term from the system of the three
  critical equations:
  \begin{align}
  \label{eq:quad}
  \sum_{i=1}^M\frac{m_i n_i}{1+n_i\theta} (
  \bar Y_i-\hat\mu(\theta))^2 &=
  r_{Y^2}(\theta) - 2\frac{r_Y(\theta)}{r_1(\theta)}
  \sum_{i=1}^M\frac{m_i n_i}{1+n_i\theta} \bar Y_i +
  \frac{r_Y(\theta)^2}{r_1(\theta)^2}\sum_{i=1}^M\frac{m_i n_i}{1+n_i\theta} \\
  \nonumber
  &=
  r_{Y^2}(\theta) - \frac{r_Y(\theta)^2}{r_1(\theta)} .
\end{align}
Solving the second equation in~\eqref{eq:kappa} with
$\mu=\hat\mu(\theta)$ for $\kappa$ thus gives
\begin{align}
  \label{eq:hat-kappa}
  \hat\kappa(\theta) &=
  \frac{N}{W + r_{Y^2}(\theta)+r_{B/m}(\theta) -
    \frac{r_Y(\theta)^2}{r_1(\theta)}}\\ 
  &=
  \frac{N r_1(\theta)}{W r_1(\theta) + r_{Y^2}(\theta)r_1(\theta) +
  r_1(\theta)r_{B/m}(\theta)- r_Y(\theta)^2}. 
\end{align}

Substituting $\hat\mu(\theta)$ and $\hat\kappa(\theta)$ into the third and last
equation in~\eqref{eq:theta}, we obtain the univariate rational equation
\begin{align}
  \label{eq:univar}
    s_{Y^2}(\theta)-2\frac{r_Y(\theta)}{r_1(\theta)}s_Y(\theta)+
    \frac{r_Y(\theta)^2}{r_1(\theta)^2}s_1(\theta)
    +s_{B/m}(\theta)
   - \frac{r_1(\theta)}{\hat\kappa(\theta) } = 0,
\end{align}
where we have divided by the non-zero rational expression
$\hat\kappa(\theta)$.  According to~(\ref{eq:hat-kappa}), this is
\begin{multline}
  \label{eq:univar2}
    s_{Y^2}(\theta)-2\frac{r_Y(\theta)}{r_1(\theta)}s_Y(\theta)+
    \frac{r_Y(\theta)^2}{r_1(\theta)^2}s_1(\theta)+s_{B/m}\\
   -
\frac{W r_1(\theta) + r_{Y^2}(\theta)r_1(\theta) +r_1(\theta)r_{B/m} -
  r_Y(\theta)^2}{N} = 0.
\end{multline}
Reexpress~(\ref{eq:univar}) in terms of the $f$ and $g$ polynomials as
\begin{multline}
  \label{eq:univar2-2}
  \frac{g_{Y^2}(\theta)}{d(\theta)^2}
  -2\frac{f_Y(\theta)}{f_1(\theta)}\frac{g_Y(\theta)}{d(\theta)^2}+
    \frac{f_Y(\theta)^2}{f_1(\theta)^2}\frac{g_1(\theta)}{d(\theta)^2}
    +\frac{g_{B/m}}{d(\theta)^2}\\
    - \frac{W f_1(\theta)d(\theta) + f_{Y^2}(\theta)f_1(\theta)+ f_1(\theta)f_{B/m}-
  f_Y(\theta)^2}{Nd(\theta)^2} = 0.
\end{multline}
The claim now follows by forming a common denominator.
\end{proof}

The denominator given in~(\ref{eq:univar4}) in
Lemma~\ref{lemma:univRatlEqn} has degree $2M+2(M-1)=4M-2$. The
numerator in (\ref{eq:univar4}) has degree
$3(M-1)+M=4M-3$; the highest degree term involves the within-group sum
of squares $W$.
The next two lemmas imply that, after cancelling common factors, the
numerator of the univariate rational function from
Lemma~\ref{lemma:univRatlEqn} has the degree claimed in the ML degree
formula from Theorem~\ref{thm:MAIN}.

\begin{lemma}\label{claim1}
  If $m_t=1$, then $(1+n_{t}\theta)$ divides the numerator of the
  rational equation~\eqref{eq:univar4}.  Hence, the polynomial
  $d_1(\theta)$ of degree $M-M_2$ divides this numerator.
\end{lemma}
\begin{lemma}\label{claim2}
  If $d_1(\theta)$ is cleared from both the numerator and the
  denominator of the rational function given in~\eqref{eq:univar4},
  then the new numerator and denominator are relatively prime
  for generic sufficient statistics $\bar Y_1, \ldots, \bar Y_M$, $W$,
  and $B_j$ with $ m_j\geq 2$.  
\end{lemma}

\begin{proof}[Proof of Lemma~\ref{claim1}]
  Let $m_{t}=1$. To show that $(1+n_{t} \theta)$ divides the
  numerator, it is sufficient to show $(1+n_{t} \theta)$ divides the
  sum of
  \begin{equation}
    \label{eq:num1a}
    N\left[f_1(\theta)^2g_{Y^2}(\theta)-2f_Y(\theta)f_1(\theta)g_Y(\theta)+
    f_Y(\theta)^2g_1(\theta)+f_1(\theta)^2g_{B/m}(\theta)\right]
  \end{equation}
  and
  \begin{equation}
    \label{eq:num1b}
    -f_1(\theta)^2[f_{Y^2}(\theta)f_1(\theta) - f_Y(\theta)^2
    +f_1(\theta)f_{B/m}(\theta)]. 
  \end{equation} 
  
  The product $f_1(\theta)^2g_{Y^2}(\theta)$ in the first
  term of~(\ref{eq:num1a}) may be rewritten as
  \begin{align*}
    &\left[ \sum_{i=1}^M m_i n_i \prod_{j \neq i} (1+n_j \theta) \right]\left[
    \sum_{k=1}^M m_k n_k \prod_{l \neq k} (1+n_l \theta) \right]\left[
    \sum_{r=1}^M m_r n_r^2 \bar{Y}_r^2\prod_{s \neq r} (1+n_s
    \theta)^2 \right] \\
    &= \sum_{i=1}^M   \sum_{k=1}^M   \sum_{r=1}^M m_i m_k m_r n_i  n_k
    n_r^2 \bar{Y}_r^2 \prod_{j \neq i} (1+n_j \theta) \prod_{l \neq k}
    (1+n_l \theta) \prod_{s \neq r} (1+n_s \theta)^2 .
  \end{align*}
  Combining this expression with the analogous expansions of the other
  three terms shows that the polynomial in (\ref{eq:num1a}) is equal
  to $N$ times
  \begin{multline}
    \label{eq:univar5}
    \sum_{i=1}^M   \sum_{k=1}^M   \sum_{r=1}^M \bigg[ (m_r
      \bar{Y}_r^2-2m_r\bar{Y}_i \bar{Y}_r + m_r \bar{Y}_i \bar{Y}_k
      +B_r) m_i m_k n_i  n_k  n_r^2 \\
      \times \prod_{j \neq i} (1+n_j
      \theta) \prod_{l \neq k} (1+n_l \theta) \prod_{s \neq r} (1+n_s
      \theta)^2 \bigg] .
  \end{multline}
  The polynomial in~\eqref{eq:num1b} can be expanded similarly.  We
  find
  \begin{align}
    \label{eq:univar6}
    &f_{Y^2}(\theta)f_1(\theta) - f_Y(\theta)^2 +f_1(\theta)f_{B/m}(\theta)\\
    =&\sum_{i=1}^M \sum_{k=1}^M (m_k \bar{Y}_i ^2-m_k \bar{Y}_i \bar{Y}_k+B_k) m_i n_i n_k  \prod_{j \neq i} (1+n_j \theta)  \prod_{l \neq k} (1+n_l \theta). \notag 
  \end{align}
  Expanding $f_1(\theta)^2$ as well, we obtain that the polynomial
  in~\eqref{eq:num1b} is equal to  
  \begin{multline}
    \label{eq:univar6-2}
    -\sum_{i=1}^M \sum_{k=1}^M \sum_{r=1}^M \sum_{u=1}^M \bigg[(m_k \bar{Y}_i
    ^2-m_k \bar{Y}_i \bar{Y}_k+B_k) m_i m_r m_u n_i n_k n_r n_u  \\
    \prod_{j \neq i} (1+n_j \theta) \prod_{l \neq k} (1+n_l
    \theta)\prod_{s \neq r} (1+n_s \theta)\prod_{v \neq u} (1+n_v
    \theta)\bigg].
  \end{multline} 
  
  Now notice that $(1+n_{t}\theta)$ divides every summand in
  (\ref{eq:univar5}) and (\ref{eq:univar6-2}) unless $i=k=r=t$ in the
  first summation, or $i=k=r=u=t$ in the second summation.  So it
  suffices to only consider these `diagonal' terms.  However, under
  the equality of indices, the quadratic expressions in the averages
  $\bar{Y}_i$ cancel.  Hence, the terms missing a factor of $(1+n_{t}
  \theta)$ in (\ref{eq:num1a}) and (\ref{eq:num1b}) sum to
  \begin{align}
    \label{eq:num3}
    B_{t} n_{t}m_{t}^2n_{t}^4\left(N- m_{t}\right) \prod_{j \neq t}
    (1+n_j \theta)^4. 
  \end{align}
  Throughout the paper, we assume that we have at least two groups
  with at least one group size $n_i \geq 2$.  
  Moreover, for generic data, $B_t=0$ if and only if $m_t=1$.  Hence,
  for generic data, the expression in~(\ref{eq:num3}) is zero if and
  only if $m_{t}=1$.  We conclude that $d_1(\theta)$ divides the
  numerator of the rational function in~\eqref{eq:univar4}.
\end{proof}

Note that the last part of the above proof shows not only that
$d_1(\theta)$ divides the numerator of~\eqref{eq:univar4}, but that
$(1+n_{t}\theta)$ does not divide the numerator when $B_t\not=0$,
which holds generically if $m_t\ge 2$.

\begin{proof}[Proof of Lemma~\ref{claim2}]
  Clearing $d_1(\theta)$ from the denominator in~\eqref{eq:univar4}
  yields the polynomial $Nd_2(\theta)d(\theta)f_1(\theta)^2$.  From
  the preceding comment, we know that $d_2(\theta)$ and the numerator
  are relatively prime for generic data $\bar Y_1, \ldots, \bar Y_M$,
  $W>0$, and $B_j>0$ with $ m_j\geq 2$.  To establish our claim, we
  will first show that $f_1 (\theta)$ does not share a common factor
  with the numerator by showing $f_1 (\theta)$ and
  $f_{Y}(\theta)^2g_1(\theta)$ to be relatively prime; all terms other
  than $f_{Y}(\theta)^2g_1(\theta)$ in the numerator of
  \eqref{eq:univar4} are multiples of $f_1(\theta)$.  Then, we will
  show that after clearing $d_1(\theta)$ in~\eqref{eq:univar4},
  $d_1(\theta)$ and the new numerator are relatively prime.
  
  Let $\theta_1,\dots,\theta_{M-1}$ be the (possibly complex) roots of
  the degree $M-1$ polynomial $f_1(\theta)$.  For each $1\le k\le
  M-1$, consider the linear form $f_{Y}(\theta_k)$ in the polynomial
  ring $\mathbb{C}[\bar Y_1,\dots,\bar Y_M]$. Let
  $V(f_{Y}(\theta_k))\subset\mathbb{C}^M$ be the zero locus of
  $f_{Y}(\theta_k)$.  Each set $V(f_{Y}(\theta_k))$ is a hyperplane of
  dimension $M-1$.  Thus, the union $\cup_{k=1}^{M-1}
  V(f_{Y}(\theta_k))$ is an $M-1$ dimensional algebraic subset of
  $\mathbb{C}^M$.  A generic vector of group means $(\bar
  Y_1,\dots,\bar Y_M)$ lies outside this lower-dimensional set, which
  means that $f_1(\theta)$ and $f_Y(\theta)$ are relatively prime for
  generic data.
  
  To show that $f_1(\theta)$ and $g_1(\theta)$ are relatively prime,
  assume $\theta_0= a + ib$ is a root of $f_1(\theta)$ and
  $g_1(\theta)$.  Since $g_1(\theta)$ is a sum of squares that is
  positive on $\mathbb{R}$, we must have $\theta_0 \notin \mathbb{R}$
  and hence $b \neq 0$.  Without loss of generality, let $n_1$ be the
  least of the group sizes $n_i$.  Rewriting $f_1(\theta_0)=0$, we get
  \begin{align}
    \label{eq:lem4}
    n_1&= - \frac{\sum_{i=2}^M m_i n_i \prod_{j\not=i} (1+n_j\theta
      _0)}{m_1\prod_{j \neq 1} (1+n_j\theta_0)} =-\sum_{i=2}^M \frac{
      m_i n_i (1+n_1 \theta_0)}{m_1 (1+n_i \theta_0)}.
  \end{align}
  The imaginary part of the right side of this equation must equal 0
  since $n_1$ is an integer.  Substituting $a+ib$ for $\theta_0$, the
  imaginary part of ~(\ref{eq:lem4}) is
  \[ 
  b \sum_{i=2}^M  \left( \frac{m_i n_i}{m_1} \right)
  \frac{(n_i-n_1)}{(1+n_ia)^2+(n_ib)^2}. 
  \]
  Since each term in the sum is positive, we obtain that $b=0$.
  Consequently, $\theta_0 \in \mathbb{R}$, which is a contradiction.
  Therefore, $f_1(\theta)$ and $g_1(\theta)$ are relatively prime.
  
  It remains to show that the numerator and denominator obtained by
  clearing the factor $d_1(\theta)$ in~\eqref{eq:univar4} are
  relatively prime for generic data.  We claim that if $m_{t}=1$ then
  $(1+n_{t}\theta)$ divides  
  \begin{equation}
    \label{eq:univar5-div-d1}
  \frac{f_1(\theta)^2g_{Y^2}(\theta)-
    2f_Y(\theta)f_1(\theta)g_Y(\theta)+
    f_Y(\theta)^2g_1(\theta)f_1(\theta)^2g_{B/m}(\theta)}{d_1(\theta)},
  \end{equation}
  while $d_1(\theta)$ and
  \begin{align}
    \label{eq:relprime1}
    Wf_1(\theta)d_2(\theta)+\frac{f_{Y^2}(\theta)f_1(\theta) -
      f_Y(\theta)^2+f_1(\theta)f_{B/m}(\theta)}{d_1(\theta)}=:
    Wf_1(\theta)d_2(\theta)+F(\theta)  
  \end{align}
  are relatively prime for generic data.
  
  The ratio in (\ref{eq:univar5-div-d1}) equals \eqref{eq:univar5}
  divided by $d_1(\theta)$.  We may rewrite  \eqref{eq:univar5} as
  \begin{multline}
    \label{eq:univar5-2}
    \sum_{i=1}^M   \sum_{k=1}^M   \sum_{r=1}^M \bigg[ (m_r
      \bar{Y}_r^2-m_r\bar{Y}_i \bar{Y}_r -m_r\bar Y_k\bar Y_r+ m_r \bar{Y}_i \bar{Y}_k
      +B_r) m_i m_k n_i  n_k  n_r^2 \\
      \times \prod_{j \neq i} (1+n_j
      \theta) \prod_{l \neq k} (1+n_l \theta) \prod_{s \neq r} (1+n_s
      \theta)^2 \bigg] .
  \end{multline}
  It is clear that the square $(1+n_t\theta)^2$ divides all terms in
  the sum \eqref{eq:univar5-2} except those for $r=i=t$ or $r=k=t$.
  However, the quadratic form in the averages $\bar Y_{i}$ vanishes
  if $r=i$ or $r=k$.  Since the terms in question have $r=t$, and
  $B_r=B_t=0$ because $m_t=1$, we conclude that $(1+n_t\theta)^2$
  divides the entire sum \eqref{eq:univar5-2}, which proves that
  $d_1(\theta)$ divides the ratio in (\ref{eq:univar5-div-d1}).  
  
  We are left to show that $d_1(\theta)$ and
  $Wf_1(\theta)d_2(\theta)+F(\theta) $ are relatively prime for
  generic data.  Let $\theta_1,\dots,\theta_{M-M_2}$ be the roots of
  $d_1(\theta)$; each root is equal to $-1/n_i$ for some index $i$.
  Since the $n_i$ are distinct, no root of $d_1(\theta)$ is a root of
  $d_2(\theta)$.  Moreover, it is easy to see that no root of
  $d_1(\theta)$ is a root of $f_1(\theta)$.  Now let $I$ be the ideal
  generated by the $M-M_2$ polynomials
  $Wf_1(\theta_k)d_2(\theta_k)+F(\theta_k)$ in the polynomial ring $
  \mathbb{C}[W, \bar Y_1, \ldots \bar Y_M, B'_1, \ldots B'_{M_2}]$,
  where the $B'_i$ stand for the between-group sums of squares $B_i$
  with multiplicity $m_i\geq 2$.  Pick sufficient statistics $W=\bar
  Y_1= \ldots = \bar Y_M \neq 0$ and $B'_1=\ldots =B'_M=0$. Since no
  root of $d_1(\theta)$ is a root of $d_2(\theta)$ or $f_1(\theta)$,
  \eqref{eq:univar6} implies that for these special data
  $Wf_1(\theta_k)d_2(\theta_k)+F(\theta_k) \neq 0$ for each $k$.  The
  zero locus $V(I)$ is thus a proper algebraic subset of
  $\mathbb{C}^{M+M_2+1}$.  Such a set is of lower dimension and, thus,
  $d_1(\theta)$ and $Wf_1(\theta)d_2(\theta)+F(\theta) $ are
  relatively prime for generic data.  
\end{proof}

%%%%%%%%%%%%%%%%%%%%%%%%%%%%%%%

\bigskip
%%%%%%%%%%%%%%%%%%%%%%
\section{Proof of formula for REML degree}
\label{sec:REMLproof}

For the proof of the REML degree formula in Theorem~\ref{thm:MAIN}, we
proceed in the same way as for the ML degree. We begin by deriving the
univariate rational function whose number of roots is the REML degree.

\begin{lemma} \label{lem:REML}
  Consider the rational function whose numerator is
  \begin{align}\label{eq:univarNum}
    &(g_1(\theta)-f_1(\theta)^2)
    [Wf_1(\theta)d(\theta)+f_{Y^2}(\theta)f_1(\theta)-f_Y(\theta)^2+f_1(\theta)f_{B/m}]
    + \\ 
    \notag &(N-1) \left[f_1(\theta)^2g_{Y^2}(\theta)-
      2f_Y(\theta)f_1(\theta)g_Y(\theta) +
    f_Y(\theta)^2g_1(\theta)+f_1(\theta)^2g_{B/m} (\theta) \right]  
  \end{align}
  and denominator is 
  \begin{align}\label{eq:univarDenom}
    {d(\theta)f_1(\theta)\left[Wf_1(\theta)d(\theta)+
        f_{Y^2}(\theta)f_1(\theta)-f_Y(\theta)^2+f_1(\theta)f_{B/m}\right]}. 
  \end{align}
  The REML degree is the degree of the numerator of this rational
  function after clearing common factors from the given numerator and
  denominator. 
\end{lemma}
\begin{proof}
The equation $\partial\bar\ell/\partial\kappa=0$ has the unique solution
\[
\hat\kappa(\theta) = \frac{N-1}{W+\sum_{i=1}^M
  \frac{m_in_i}{1+n_i\theta}(\bar Y_i-\hat\mu(\theta))^2
  +\sum_{i=1}^M \frac{n_i}{1+n_i \theta} B_i};
\]
compare~(\ref{eq:reml-partial-kappa}).  Substituting
$\hat\kappa(\theta)$ into the partial derivative
$\partial\bar\ell/\partial\theta$ yields the univariate function
\begin{align}
  \label{eq:univarREML}
  -\sum_{i=1}^M \frac{m_in_i}{1+n_i\theta} &+ \frac{\sum_{i=1}^M
    \frac{m_in_i^2}{(1+n_i\theta)^2}}{\sum_{i=1}^M
    \frac{m_in_i}{1+n_i\theta}}  \\ 
  \notag  & +\hat\kappa(\theta) \left[ \sum_{i=1}^M
    \frac{m_in_i^2}{(1+n_i\theta)^2}(\bar Y_i
    -\hat\mu(\theta))^2+\sum_{i=1}^M \frac{n_i^2}{(1+n_i \theta)^2}
    B_i\right]=0; 
\end{align}
recall~(\ref{eq:reml-partial-theta}).  We can now simplify and rewrite
\eqref{eq:univarREML}, forming a common denominator, to obtain the
desired rational function.
\end{proof}

The degree of the numerator in Lemma~\ref{lem:REML} is $4M-3$, but it
shares common factors with the denominator.  In fact, in the proof of
Lemma~\ref{claim1}, we have shown that $d_1(\theta)$ divides
$f_{Y^2}(\theta)f_1(\theta)-f_Y(\theta)^2+f_1(\theta)f_{B/m}$. Thus,
$d_1(\theta)^2$ divides the denominator from~\eqref{eq:univarDenom}.
To prove Theorem \ref{thm:MAIN}, it remains to prove the following two
facts.

\begin{lemma}\label{claim1reml}
  The polynomial $d_1(\theta)^2$ divides the numerator
  \eqref{eq:univarNum}.
\end{lemma}
\begin{lemma}\label{claim2reml}
  After clearing $d_1(\theta)^2$ from \eqref{eq:univarNum} and
  \eqref{eq:univarDenom}, the new numerator and new denominator are
  relatively prime for generic data.
\end{lemma} 

\begin{proof}[Proof of Lemma \ref{claim1reml}]
  From the proof of Lemma~\ref{claim1}, we know that $d_1(\theta)$
  divides the polynomial
  $f_{Y^2}(\theta)f_1(\theta)-f_Y(\theta)^2+f_1(\theta)f_{B/m}$.
  Moreover, as shown in the proof of Lemma~\ref{claim2}, the square
  $d_1(\theta)^2$ divides
  \[
  f_1(\theta)^2g_{Y^2}(\theta)-2f_1(\theta)f_Y(\theta)g_Y(\theta) +
  f_Y(\theta)^2g_1(\theta)+f_1(\theta)^2g_{B/m} (\theta).
  \]
  To complete the proof of the present lemma, it suffices to show that
  $d_1(\theta)$ divides $g_1(\theta)-f_1(\theta)^2$.  However, with
  some distributing and grouping, we see
  \begin{align}   
    &g_1(\theta) -f_1(\theta)^2 = \notag\\
    &=\sum_{i=1}^M m_i n_i^2 \prod_{j\not=i}
    (1+n_j\theta)^2-\sum_{i=1}^M \sum_{k=1}^M m_i m_k n_i n_k
    \prod_{j\not=i} 
    (1+n_j\theta)\prod_{l\not=k}(1+n_j\theta) \notag \\
    &=\sum_{i=1}^M (m_i -m_i^2) \prod_{j \neq i}
    (1+n_j\theta)-\sum_{i=1}^M \sum_{k>i}^M 2n_i n_k \prod_{j\not=i}
    (1+n_j\theta)\prod_{l\not=k}(1+n_j\theta), \notag 
  \end{align}
  which is divisible by $(1+n_{t}\theta)$ if and only if $m_{t}=1$. 
\end{proof}

\begin{proof}[Proof of Lemma \ref{claim2reml}]
  We first show that if $m_{t} \geq 2$, then, for generic data,
  $(1+n_{t}\theta)$ and the numerator from~\eqref{eq:univarNum} are
  relatively prime.  Consider
  \begin{multline}
    \label{eq:remlnum1}
    (g_1(\theta)-f_1(\theta)^2)\left[f_{Y^2}(\theta)f_1(\theta) -
    f_Y(\theta)^2 +f_1(\theta)f_{B/m}(\theta)\right]+\\ 
    (N-1)[f_1(\theta)^2g_{Y^2}(\theta)-2f_Y(\theta)f_1(\theta)g_Y(\theta)+
    f_Y(\theta)^2g_1(\theta)+g_{B/m}(\theta)f_1(\theta)^2]. 
  \end{multline}  
  Using the results from the proof of Lemma~\ref{claim1} and writing
  out the involved summations, \eqref{eq:remlnum1} is seen to be equal
  to
  \begin{multline}
    \label{eq:num2}
    \left( \sum_{i=1}^M \sum_{k=1}^M \sum_{r=1}^M  (m_k \bar{Y}_i
    ^2-m_k \bar{Y}_i \bar{Y}_k+B_k) m_i m_r  n_i n_k n_r^2  \right. \\ 
    \left.\quad\quad\quad \prod_{j \neq i} (1+n_j \theta) \prod_{l
    \neq k} (1+n_l \theta)\prod_{s \neq r} (1+n_s \theta)^2  \right) 
    \\  
    -\left(  \sum_{i=1}^M \sum_{k=1}^M \sum_{r=1}^M \sum_{u=1}^M (m_k
    \bar{Y}_i ^2-m_k \bar{Y}_i \bar{Y}_k+B_k) m_i m_r m_u n_i n_k n_r
    n_u \right. \\ 
    \quad\quad\quad \left. \prod_{j \neq i} (1+n_j \theta)  \prod_{l
    \neq k} (1+n_l \theta)\prod_{s \neq r} (1+n_s \theta)\prod_{v \neq
    u} (1+n_v \theta)\right)\\     
    +(N-1)\left[\sum_{i=1}^M   \sum_{k=1}^M   \sum_{r=1}^M (m_r
    \bar{Y}_r^2-2m_r\bar{Y}_i \bar{Y}_r + m_r \bar{Y}_i \bar{Y}_k
    +B_r) m_i m_k n_i  n_k  n_r^2  \right. \\ 
    \left.  \prod_{j \neq i} (1+n_j \theta) \prod_{l \neq k}
    (1+n_l \theta) \prod_{s \neq r} (1+n_s \theta)^2 \right] . 
  \end{multline} 
  The factor $(1+n_{t}\theta)$ divides every summand in the above
  summations unless $t=i=k=r=u$, so it suffices to only consider these
  terms. Letting $t=i=k=r=u$, the terms missing a factor of $(1+n_{t}
  \theta)$ sum to a term we already encountered, namely, that
  in~(\ref{eq:num3}). 
  The discussion following display (\ref{eq:num3}) shows that if the
  data is generic and $m_{t} \geq 2$, then $(1+n_i \theta)$ does not
  divide the numerator given in~\eqref{eq:univarNum}.
  
  Continuing to work through the factors of the denominator
  from~\eqref{eq:univarDenom}, assume that $\theta_0$ is a root of
  $f_1(\theta)$.  Then everything vanishes in the numerator except for
  two terms $-g_1(\theta_0)f_Y(\theta_0)^2$ and
  $(N-1)f_Y(\theta_0)^2g_1(\theta_0)$, which add to
  $(N-2)f_Y(\theta_0)^2g_1(\theta_0)$. From the proof of
  Lemma~\ref{claim2}, we know $f_Y(\theta_0)^2g_1(\theta_0) \neq 0$
  for generic data, so since we are working under the assumption of at
  least two groups and at least one group size $n_i\geq2$, the
  numerator and $f_1(\theta)$ are relatively prime for generic data.
  
  Finally, we need to show
  \[
  H(\theta)
  :=f_1(\theta)^2g_{Y^2}(\theta)-2f_1(\theta)f_Y(\theta)g_Y(\theta) +
  f_Y(\theta)^2g_1(\theta)+f_1(\theta)^2g_{B/m}(\theta)
  \]
  and
  \[
  G(\theta)
  :=Wf_1(\theta)d_2(\theta)+F(\theta),
  \]
  are relatively prime for generic data $W$, $\bar Y_1, \ldots \bar
  Y_M$, and $B_i$ with $m_i\geq 2$; the polynomial $F(\theta)$ was
  defined in~(\ref{eq:relprime1}).  We will again denote the
  between-group sums of squares with multiplicities $m_i\ge 2$ as
  $B'_1, \ldots, B'_{M_2}$.  By a standard algebraic results, the
  polynomials $G(\theta)$ and $H(\theta)$ share a common root $\theta$
  if and only if a certain polynomial in their coefficients vanishes;
  this polynomial is called the resultant and we denote it by
  $\text{Res} (G,H)$.  Since both $H(\theta)$ and $G(\theta)$ have
  coefficients that are polynomials in the sufficient statistics $W$,
  $\bar Y_1, \ldots \bar Y_M$, and $B'_1,\dots,B'_{M_2}$, we may
  regard $\text{Res} (G,H)$ as a polynomial in the ring $\mathbb{C}[W,
  \bar Y_1, \ldots \bar Y_M, B'_1, \ldots, B'_{M_2}]$. By
  Lemma~\ref{claim2}, for any given generic choice of $\bar Y_1,
  \ldots, \bar Y_M, B'_1, \ldots, B'_{M_2}$, a root $\theta_0$ of $H$
  is not a root of $f_1(\theta)$ or $d_2(\theta)$.  Hence, $\theta_0$
  is a root of $G$ if and only if 
  \begin{equation}
    \label{eq:specialW}
    W=-\frac{F(\theta_0)}{d_2(\theta_0) f_1(\theta_0)}.
  \end{equation}
  Picking $W$ not to satisfy~(\ref{eq:specialW}) shows that
  $\text{Res} (G,H)$ is not the zero polynomial in $\mathbb{C}[W, \bar
  Y_1, \ldots \bar Y_M, B'_1, \ldots, B'_{M_2}]$.  Hence, the zero
  locus of $\text{Res} (G,H)$ is a set of lower dimension, and we
  conclude that $H$ and $G$ are relatively prime for generic data.
\end{proof}

\section{General mean structure in the one-way layout}
\label{sec:general-mean}

The one-way layout as specified in~(\ref{eq:model-oneway}) postulates
a common mean $\mu$ for all observations $Y_{ij}$.  Often the interest
is instead in a more general mean space.  Formally, consider the model
\begin{equation}
  \label{eq:model-general-mean}
  Y_{ij} = \mu_{ij}+\alpha_i + \varepsilon_{ij}, \qquad i=1,\dots,q, \quad
  j=1,\dots,n_i,
\end{equation}
where the random effects $\alpha_i\sim\mathcal{N}(0,\tau)$ and the
error terms $\varepsilon_{ij}\sim\mathcal{N}(0,\omega)$ are again all
mutually independent.  However, the array of means $(\mu_{ij})$ may
now belong to a linear subspace of $\mathbb{R}^N$ that we assumed to
be spanned by the independent columns of a full rank design matrix
$X\in\mathbb{R}^{N\times p}$; as before, $N=n_1+\dots+n_q$ is the
sample size.  In other words,
\begin{equation}
  \label{eq:model-general-mean2}
  \text{vec}(\mu_{ij}) =X \beta
\end{equation}
for some unknown (fixed) mean parameter vector $\beta\in\mathbb{R}^p$.

ML and REML estimation with more general mean structure can be
approached algebraically in the exactly the same way as before.  It is
convenient to reparametrize the covariance matrix in terms of
$\kappa=1/\omega$ and $\theta=\tau/\omega$.  For known covariance
parameters, the ML estimate $\hat\beta(\theta)$ of $\beta$ is obtained
by generalized least squares and depends on $\theta$ but not on
$\kappa$.  For fixed $\theta$, it is then also straightforward to
solve the ML or REML equations for $\kappa$.  This way we may reduce
algebraic solution of the likelihood equations to solving a single
rational equation in $\theta$.  In this section we demonstrate that
the involved algebraic computations are feasible in a larger example.
Before going into the details of the example, we would like to offer
the following conjecture based on numerical experiments with smaller
models and randomly chosen design matrices.  It states that the ML and
REML degrees for the model specified by (\ref{eq:model-general-mean})
and (\ref{eq:model-general-mean2}) cannot exceed the largest possible
respective degrees in the model with common mean $\mu$.  Recall that
the largest degrees arise in the entirely unbalanced case with group
sizes $n_1,\dots,n_q$ that are pairwise distinct.

\begin{conj}
  \label{conj:general-mean}
  For any design matrix $X\in\mathbb{R}^{N\times p}$ that has the vector
  $(1,\dots,1)^T$ in its column span $\text{span}(X)$, the ML degree
  for the one-way layout with mean space $\text{span}(X)$ and $q$
  random group effects is bounded above by $3q-3$.  Similarly, the
  REML degree is bounded above by $2q-3$.
\end{conj}

According to this conjecture, the degrees would grow only linearly
with the number of groups, which would suggest that a moderately large
number of unbalanced groups can be handled in algebraic computations.

\begin{example}
  \label{ex:atkinson}
  With the goal of providing linguistic support for an African origin
  of modern humans, Atkinson \cite{Atkinson:2011} fits regression
  models to data on the phonemic diversity of languages.  The data,
  which can be obtained from the journal's online supplementary
  material, concern $N=504$ languages that are classified into $q=109$
  language families.  Besides quantitative summary measures of
  phonemic diversity, the available information includes the size of
  the population speaking each language and the distance between a
  chosen center for each language and an inferred origin in Africa,
  the latter being the main covariate of interest.
  
  One model of interest in this application is a one-way layout with
  groups corresponding to the language families.  The response
  $Y_{ij}$ is the phonemic diversity of the $j$th language in the
  $i$th family, which, as in (\ref{eq:model-general-mean}) and
  (\ref{eq:model-general-mean2}), is modelled as
  \begin{equation}
    \label{eq:atkinson}
    Y_{ij} = \beta_0+\beta_1 \log(\textit{P}_{ij}) + \beta_2 \textit{D}_{ij} +
  \alpha_i +\epsilon_{ij}, \qquad i=1,\dots,q, \quad
  j=1,\dots,n_i.
  \end{equation}
  Here, $\textit{P}_{ij}$ stands for the population size and
  $\textit{D}_{ij}$ is the distance from the origin in Africa.  As can
  be expected, the data is unbalanced.  The group sizes
  $n_1,\dots,n_{109}$ fall into the range from 1 to 62.  There are
  $M=17$ distinct group sizes of which $M_2=9$ have multiplicity two
  or larger.  Hence, by Theorem~\ref{thm:MAIN}, the one-way layout
  with all means equal has ML degree 57 and REML degree 49.  However,
  as we show next, the mean structure can affect the ML and REML
  degree.
  
  Computations we did using the software Maple show that the ML
  degree of the model given by~(\ref{eq:atkinson}) is 83, whereas the
  REML degree is 71.  Exact computations in analogy to the ones given
  in Example~\ref{ex:dyestuff} produce large integer coefficients, too
  large to display on paper but easily handled by a computer.  Solving
  the polynomial equations for ML and REML numerically, each equation
  is seen to have a unique positive root, namely,
  \begin{equation}
    \label{eq:atkinson-algebra}
    \hat\theta_{\text{ML}} \approx   0.3706 \quad\text{and}\quad
    \hat\theta_{\text{REML}} \approx 0.3853.
  \end{equation}
  Each root gives a local and, thus, global maximum of the concerned
  likelihood function.  We remark that the ML equation has twelve
  negative real roots.  The REML equation has no other real roots.
  Running the numerical optimizers implemented in the R package {\tt
    lme4} yields estimates that agree with (\ref{eq:atkinson-algebra})
  in all the given digits.  As in Example~\ref{ex:dyestuff}, the fact
  that our two univariate polynomials each have a unique positive root
  manifests itself in a single sign change in the coefficient
  sequence.  Finally, we remark that when omitting either the
  covariate $\log(\textit{P})$ or the covariate $\textit{D}$, the ML
  degree drops to 72 and the REML degree drops to 61.
\end{example}

\section{Balanced two-way layouts}
\label{sec:two-way}

Suppose we have observations $Y_{ijk}$ that are cross-classified
according to two factors and model the observations in an additive
two-way layout as
\begin{align}
  \label{eq:two-way-additive}
  Y_{ijk} &= \mu +\alpha_i+\beta_j +\epsilon_{ijk}, \\
  & \qquad\qquad
  i=1,\dots,r,\quad j=1,\dots,q, \quad
  k=1,\dots,n.  \notag
\end{align}
The terms $\alpha_i\sim\mathcal{N}(0,\tau_1)$ and
$\beta_j\sim\mathcal{N}(0,\tau_2)$ are normally distributed random
effects.  The error terms are distributed as
$\epsilon_{ijk}\sim\mathcal{N}(0,\omega)$, and all the random
variables $\alpha_i$, $\beta_j$ and $\epsilon_{ijk}$ are mutually
independent.  Finally, there is one (fixed) mean parameter
$\mu\in\mathbb{R}$.  A related model is obtained by including random
interaction terms $\gamma_{ij}\sim\mathcal{N}(0,\tau_{12})$ in the
defining equations
\begin{align}
  \label{eq:two-way-interaction}
  Y_{ijk} &= \mu +\alpha_i+\beta_j+\gamma_{ij} +\epsilon_{ijk}, \\
  & \qquad\qquad
  i=1,\dots,r,\quad j=1,\dots,q, \quad
  k=1,\dots,n.  \notag
\end{align}
The interaction terms $\gamma_{ij}$ are again mutually independent and
independent of all other random variables appearing on the right hand
side of (\ref{eq:two-way-interaction}).  

The models in (\ref{eq:two-way-additive}) and
(\ref{eq:two-way-interaction}) are balanced; the groups of
observations $Y_{ij1},\dots,Y_{ijn}$ specified by the different index
pairs $(i,j)$ are all of size $n$.  It is known that REML leads to
closed form estimates for each of the two balanced models; compare
\cite{Hocking:1985,Searle:1992,Sahai:I}.  In other words, the REML
degree of either model is one.  ML estimation, however, presents a
non-trivial algebraic problem.  The ML degree can be derived using
Gr\"obner basis calculations, and we see that balanced two-way layouts
have closed form ML estimates in the sense of Cardano's formula.

\begin{theorem}
  \label{thm:balanced-two-way}
  The ML degree of balanced additive two-way layout with random effects
  is four.  The same holds for the model with random interaction.
\end{theorem}
\begin{proof}
  Define the sum of squares
  \begin{align*}
    \textit{SSA} &= \sum_{i=1}^r qn(\bar Y_{i\centerdot \centerdot} -
    \bar Y_{\centerdot\centerdot\centerdot})^2,\\ 
    \textit{SSB} &= \sum_{j=1}^r rn(\bar Y_{\centerdot j\centerdot } -
    \bar Y_{\centerdot \centerdot \centerdot })^2,\\ 
    \textit{SSAB} &= \sum_{i=1}^r \sum_{j=1}^q 
    n(\bar Y_{ij\centerdot } -\bar Y_{i\centerdot \centerdot } - \bar Y_{\centerdot j\centerdot }
    +\bar Y_{\centerdot \centerdot \centerdot })^2,\\
    \textit{SSE} &= \sum_{i=1}^r \sum_{j=1}^q\sum_{k=1}^n
    ( Y_{ijk} -\bar Y_{ij\centerdot })^2,
  \end{align*}
  where we use the convention that the overbar indicates that an
  average was formed and the `$\centerdot$' subscripts specify which
  indices were averaged over.
  
  {\em (No interaction)\/} The ML equations for the additive model
  given by~(\ref{eq:two-way-additive}) are derived, for instance, in
  Chapter 4.7.d of \cite{Searle:1992} and in Chapter 3 of
  \cite{Sahai:I}.  One equation leads to the ML estimator
  \[
  \hat\mu = \bar Y_{\centerdot \centerdot \centerdot }.
  \]
  The rational equations for the 
  variance components may be written as
  \begin{align}
    \frac{rqn-r-q+1}{\omega} - \frac{1}{\omega+qn\tau_1+rn\tau_2} &=  
    \frac{\textit{SSAB}+\textit{SSE}}{\omega^2},\\
    \frac{r-1}{\omega+qn\tau_1} + \frac{1}{\omega+qn\tau_1+rn\tau_2} &=  
    \frac{\textit{SSA}}{(\omega+qn\tau_1)^2},\\
    \frac{q-1}{\omega+rn\tau_2} + \frac{1}{\omega+qn\tau_1+rn\tau_2} &=  
    \frac{\textit{SSB}}{(\omega+rn\tau_2)^2}.
  \end{align}
  Clearing the common denominators $\omega^2$, $(\omega+qn\tau_1)^2$,
  $(\omega+rn\tau_2)^2$, and $\omega+qn\tau_1+rn\tau_2$ gives a
  polynomial equation system.  However, multiplying each equation with
  the relevant product of these denominators introduces new solutions
  that are not solutions of the original rational equations.  Using
  saturation as explained in Chapter 2 of \cite{Drton:book}, we can
  remove these extraneous solutions and obtain a polynomial equation
  system of degree 4.  (We remark that software such as Maple is
  able to produce a lexicographic Gr\"obner basis over the field of
  fractions in $r$, $q$, $n$, and the four sums of squares.)
  
  {\em (With interaction)\/} Chapter 4.7.d of \cite{Searle:1992} also
  gives the ML equations for the model with interaction defined
  by~(\ref{eq:two-way-interaction}); see also Chapter 4 of
  \cite{Sahai:I}.  Two equations determine the ML estimators
  \[
  \hat\mu = \bar Y_{\centerdot \centerdot \centerdot }, \qquad
  \hat\omega = \frac{\textit{SSE}}{rq(n-1)}.
  \]
  The rational equations for the remaining variance components can be
  written as
  \begin{align}
    \frac{(r-1)(q-1)}{\hat\omega+n\tau_{12}} -
    \frac{1}{\hat\omega+qn\tau_1+rn\tau_2+n\tau_{12}} &=   
    \frac{\textit{SSAB}}{(\hat\omega+n\tau_{12})^2},\\
    \frac{r-1}{\hat\omega+qn\tau_1+n\tau_{12}} +
    \frac{1}{\hat\omega+qn\tau_1+rn\tau_2+n\tau_{12}} &=   
    \frac{\textit{SSA}}{(\hat\omega+qn\tau_1+n\tau_{12})^2},\\
    \frac{q-1}{\hat\omega+rn\tau_2+n\tau_{12}} +
    \frac{1}{\hat\omega+qn\tau_1+rn\tau_2+n\tau_{12}} &=   
    \frac{\textit{SSB}}{(\hat\omega+rn\tau_2+n\tau_{12})^2}.
  \end{align}
  Clearing the denominators carefully via saturation yields a
  polynomial equation system of degree 4.  (Again, a lexicographic
  Gr\"obner basis can be obtained with $r$, $q$, $n$ and the sums of
  squares as parameters to the equations.)
\end{proof}

We briefly illustrate algebraic computation of the ML estimators in an example
that involves the additive two-way layout.

\begin{example}
  \label{ex:penicillin}
  The R package {\tt lme4} contains data from experiments for an
  assessment of the variability between samples of penicillin.  The
  data are described in detail in \cite{Davies:1972}.  The response is
  a diameter measurement of the zone in which growth of an organism
  is inhibited by the penicillin.  The experiments are
  cross-classified according to the assay plate and the penicillin
  sample used.  The former is a factor with $r=24$ levels, the latter
  has $q=6$ levels. There are no replications to be considered in this
  case, that is, $n=1$.  We will consider the additive model for which
  the relevant sums of squares are
  \begin{align*}
  \textit{SSA}&= 105 \tfrac{8}{9}, &  
  \textit{SSB}&= 449 \tfrac{2}{9}, &  
  \textit{SSAB}+\textit{SSE} &= 34 \tfrac{7}{9}.
  \end{align*}
  
  Using the saturation computation alluded to in the proof of
  Theorem~\ref{thm:balanced-two-way}, we obtain the 
  polynomial equation system
  \begin{align*}
    204808595904\, \omega^4 -1801205257140\, \omega^3+2545119731943\,
    \omega^2-1070402996440\, \omega\\  
    \notag
    +139045932165 &=0,\\ 
    2481278604010272\, \tau_1+507582172417738176\,
    \omega^3 -4309720916424828084\, \omega^2 \quad\\
    \notag
    +
    4998133978544934251\, \omega-1133204709683307975 &=0,\\ 
    2481278604010272\, \tau_2+
    534435082556924736\, \omega^3-4538697213124439100\,
    \omega^2 \quad\\
    \notag +5270402449572117709\, \omega-1201351121037374475&=0. 
  \end{align*}
  This polynomial system has the same solution set as the original
  rational ML equations.  The polynomials on the left hand sides of
  the equations form a lexicographic Gr\"obner basis and are readily
  solved.  First, solve the quartic equation in $\omega$.  Next, plug
  each of the four solutions for $\omega$ into the other two equations
  and solve the resulting linear equations for $\tau_1$ and $\tau_2$,
  respectively.  In the present example, all four solutions are real
  but only one is feasible with $\omega,\tau_1,\tau_2\ge 0$.  This
  solution is
  \begin{align*}
  \hat\omega &= 0.302425, & 
  \hat\tau_1 &= 0.714992, & 
  \hat\tau_2 &= 3.135188.
  \end{align*}
  It defines the unique global maximum of the likelihood function.
\end{example}

\section{Conclusion}
\label{sec:conclusion}

This paper takes a first step towards understanding the algebraic
complexity of ML and REML estimation in linear mixed models.  Our main
results in Theorem~\ref{thm:MAIN} concern the unbalanced one-way
layout with common mean for all observations.  It would be interesting
to generalize the results to one-way classifications with more
complicated mean spaces; recall Conjecture~\ref{conj:general-mean}.
Similarly, it would be interesting to study unbalanced two- and
higher-way layouts, although these models would require more
sophisticated mathematical treatment because it is no longer possible
to analyze a single univariate rational equation; compare
Section~\ref{sec:two-way}.

A remarkable feature common to Examples~\ref{ex:dyestuff}
and~\ref{ex:atkinson} is that Descartes' rule of sign applied to a
univariate polynomial in the variance ratio $\theta$ reveals that
there is a unique feasible solution to the ML/REML equations.  The
same was true for many other examples of unbalanced one-way
classifications that we computed.  This said, we also saw cases with
more than one sign change and the number of positive solutions for
$\theta$ not matching up with the sign changes.

To our knowledge, the literature does not supply many examples of
linear mixed models with multimodal likelihood functions.  We conclude
by giving two simulated examples that demonstrate the mathematical
possibility of more than one mode.  Such examples were rare in our
simulations, which is in agreement with findings of
\cite{Swallow:1984} who also treat the unbalanced one-way layout.
While uniqueness of local optima is not explicitly discussed in
\cite{Swallow:1984}, the authors remark in their conclusion that
``varying the iteration starting point slightly affects the rate of
convergence, but not the [mean square errors] or biases of the
[ML and REML] estimators.''  The examples we give involve three
positive roots to the ML or REML equations for the variance ratio
$\theta$.  We do not know of examples with more positive roots.

\begin{example}
  \label{ex:bimodal-ML}  
  Consider the one-way layout with a single grand mean $\mu$ from
  (\ref{eq:model-oneway}).  Take $q=5$ groups of sizes
  \begin{align*}
    n_1&=2, &
    n_2&=5, &
    n_3&=10, &
    n_4&=20, &
    n_5&=50.
  \end{align*}
  Let the sufficient statistics be the five group averages
  \begin{equation*}
    \arraycolsep 0.1cm
    \begin{array}{rcrlr@{\qquad\qquad}rcrlr} 
      \bar Y_1 &=& -\tfrac{73571}{14273} &\approx& -5.1546,&
      \bar Y_2 &=& \tfrac{13781}{78326} &\approx &0.1759, \\[0.2cm]
      \bar Y_3 &=& -\tfrac{13277}{92152} &\approx& -0.1441,&
      \bar Y_4 &=& \tfrac{31207}{202567} &\approx& 0.1541, \\[0.2cm]
      \bar Y_5 &=& -\tfrac{15713}{24121} &\approx& -0.6514,
    \end{array}
  \end{equation*}
  and the within-group sum of squares
  \[
  W = \tfrac{116487}{421} \approx 276.69.
  \]
  The univariate ML equation in $\theta$ has three nonnegative
  solutions, namely,
  \[
  \hat\theta_{\text{ML},1}\approx 0.00838738, \qquad
  \hat\theta_{\text{ML},2}\approx 0.118458, 
  \qquad \hat\theta_{\text{ML},3}\approx 0.338944;
  \]
  having specified six digits we should add that the solutions were
  computed treating the above rational fractions as the input.  The
  solution $\hat\theta_{\text{ML},1}$ yields the global maximum of the
  likelihood function, whereas $\hat\theta_{\text{ML},2}$ and
  $\hat\theta_{\text{ML},3}$ determine a saddle point and local
  maximum, respectively.  In contrast, the restricted likelihood
  function has a unique local and global maximum for
  \[
  \hat\theta_{\text{REML}}\approx 0.771763.
  \]
  The data was simulated from the model with mean $\mu_0=0$, and
  variance components $\tau_0=3$ and $\omega_0=2$, which gives
  $\theta_0=3/2$.
\end{example}

\begin{example}
  \label{ex:bimodal-REML}  
  Continuing with the setup from Example~\ref{ex:bimodal-ML}, change the
  sufficient statistics to
  \begin{equation*}
    \arraycolsep 0.1cm
    \begin{array}{rcrlr@{\qquad\qquad}rcrlr} 
      \bar Y_1 &=& \tfrac{230081}{40206} &\approx& 5.7226,&
      \bar Y_2 &=& \tfrac{721282}{5630371} &\approx &0.1281, \\[0.2cm]
      \bar Y_3 &=& \tfrac{29305}{95646} &\approx& 0.3064,&
      \bar Y_4 &=& \tfrac{15365}{37988} &\approx& 0.4045, \\[0.2cm]
      \bar Y_5 &=& -\tfrac{569}{40932} &\approx& -0.0139,
    \end{array}
  \end{equation*}
  and 
  \[
  W = \tfrac{755002}{1759} \approx 429.22.
  \]
  Now, all real solutions to the ML equations are negative.  Thus, the
  global maximum of the likelihood is achieved at the boundary point
  $\hat\theta_{\text{ML}}=0$.  In contrast, the REML equations have three
  feasible solutions for $\theta$, namely,
  \[
  \hat\theta_{\text{REML},1}\approx 0.00492193,\qquad
  \hat\theta_{\text{REML},2}\approx 0.159465,
  \qquad \hat\theta_{\text{REML},3}\approx 0.2414611.
  \]
  The solution $\hat\theta_{\text{REML},1}$ gives the global maximum
  of the restricted likelihood function.  The solutions
  $\hat\theta_{\text{REML},2}$ and $\hat\theta_{\text{REML},3}$
  determine a saddle point and a local maximum, respectively.  The
  data was simulated as in Example~\ref{ex:bimodal-ML}.
\end{example}

Readers experimenting with the two examples just given will find the
likelihood functions to be rather flat between the three stationary
points, which give log-likelihood values that differ by less than 0.1.

In both Example~\ref{ex:bimodal-ML} and Example~\ref{ex:bimodal-REML},
the first group is of the smallest size but has group mean that is
largest in absolute value.  The other means are comparatively close to
each other.  We experimented with permuting the means, while holding
the group sizes fixed.  In Example~\ref{ex:bimodal-REML}, eight out of
120 permutations give bimodal restricted likelihood functions.  Two
permutations yield three positive roots to the REML equations.  The
other six cases have two positive roots, and one of the two local
maxima occurs for $\theta=0$.  The eight permutations generate the
group of permutations that keep the first mean fixed.  In this
example, there is clearly negative correlation between the group sizes
$n_i$ and the group means $\bar Y_i$.  (In practice, such dependence
could arise from selection effects.)  The eight permutations of
interest turn out to give the eight most negative correlations between
group sizes and means.  In similar experiments for
Example~\ref{ex:bimodal-ML}, which features positive correlation
between group sizes and means, bimodal likelihood functions are
obtained for 18 permutations.  Again, these permutations keep the
first mean fixed.  Only three permutations give three positive roots
to the ML equations.  The 18 permutations include the top six
permutations in terms of large positive correlation but also
the permutation whose associated correlation ranks 43rd.

While dependence between group means and sizes plays a role in
Examples~\ref{ex:bimodal-ML} and~\ref{ex:bimodal-REML}, the precise
interplay between them appears to be subtle.  For instance, when
varying the mean $\bar Y_1$ in Example~\ref{ex:bimodal-ML} and keeping
all other sufficient statistics fixed, we find that there are three
positive roots to the ML equations when $-5.47\le \bar Y_1\le -5.08$
but a unique root otherwise; we experimented with a grid of values in
$[-10,10]$.  In particular, the likelihood function is unimodal for
larger negative values of $\bar Y_1$.  It would be interesting, but
presumably difficult, to get a better understanding of the
semi-algebraic set of sufficient statistics that give (restricted)
likelihood functions with more than one local maximum.

\bibliographystyle{amsalpha}
\bibliography{randeff}

%%%%%%%%%%%%%%%%%%%%%%
\end{document}